\documentclass{siamltex}
\usepackage{amsfonts,amsmath,amssymb,color,verbatim}
\usepackage{stmaryrd}
\usepackage{hyperref}
\usepackage[mathscr]{eucal}
\usepackage{accents}
\usepackage{pgfplots}
\usepackage{ifoddpage}
\usepackage{marginnote}
\usepackage{float} %Floats
\usepackage{placeins} %Floats!!
\usetikzlibrary{patterns}
\usepackage{wasysym}
\usepackage[ruled,vlined]{algorithm2e}
\usepackage[colorinlistoftodos,textwidth=4cm,shadow]{todonotes}
\usepackage{bm}
\usepackage{tikz}
\usepackage{enumitem}
\usepackage{amsmath}
\usepackage{pgf}
\usepackage{multirow}
\usepackage{caption}
\usepackage{subcaption}
\usepackage{url}
\makeatletter
\def\BState{\State\hskip-\ALG@thistlm}
\makeatother
\newcounter{Ivan}

\newcounter{Dasha}

\def\len{\mathop{\mathrm{len}}\nolimits}
\def\SE{\mathop{\mathrm{SE}}\nolimits}
\def\sons{\mathop{\mathrm{sons}}\nolimits}

\title{Preconditioners for hierarchical matrices based on their extended sparse form} 
\author{Daria Sushnikova\footnotemark[5] \and Ivan V.~Oseledets\footnotemark[3]\ \footnotemark[5]}
\begin{document}

\maketitle

\renewcommand{\thefootnote}{\fnsymbol{footnote}}

\footnotetext[3]{Skolkovo Institute of Science and Technology,
Novaya St.~100, Skolkovo, Odintsovsky district, 143025 
Moscow Region, Russia (i.oseledets@skolkovotech.ru)}
\footnotetext[5]{Institute of Numerical Mathematics Russian Academy of Sciences,
Gubkina St. 8, 119333 Moscow, Russia}

\renewcommand{\thefootnote}{\arabic{footnote}}

\begin{abstract} 
    In this paper we consider linear systems with dense-matrices which arise from numerical solution of boundary integral equations. Such matrices can be well-approximated with $\mathcal{H}^2$-matrices. We propose several new preconditioners for such matrices that are based on the equivalent \emph{sparse extended form} of $\mathcal{H}^2$-matrices. In the numerical experiments we show that the most efficient approach is based on the so-called reverse-Schur preconditioning technique.  
\end{abstract}



\pagestyle{myheadings} \thispagestyle{plain}

\section{Introduction}

Many physical problems in acoustics, electrostatics \cite{Sent-electr_ex-1992} and some other \cite{white-schur-2009} areas lead to boundary and volume integral equations with non-local kernels. 
Numerical solution of integral equations is challenging, since even the computation of all the matrix elements is offen not possible
for the problems of interest. Fortunately, the matrices, arising from the discretization of integral equations can be approximated 
well with \emph{data-sparse matrices}. Among the most important approaches are hierarchical matrices ($\mathcal{H}$-matrices) \cite{Borm-h-2003}, mosaic-skeleton method \cite{tee-mosaic-1996} , and  hierarchical semiseparable matrices (HSS-matrices ) \cite{ShDewilde-hss-2007,Mar-hss-2011, ChLi-hss-2007, ChLyons-hss-2005}.   All of those classes of matrices correspond to the idea of block low-rank approximation and have their roots in the classical fast multipole method (FMM) \cite{GrRo-fmm-1987,GrRo-fmm-1988,Ro-fmm-1985}. In this paper we consider linear systems with $\mathcal{H}^2$-matrices \cite{HackBorm-h2-2002, Borm-h2-2010}. 
An $\mathcal{H}^2$-matrix can be multiplied by a vector fast, thus iterative methods can be used to solve linear systems. But for efficient solution of a linear system it is not enough. If matrix is ill-conditioned we have to use preconditioners. A very efficient approach is based on the approximate factorization of hiearchical matrices \cite{Beb-hlu-2005}. Algorithms with almost linear complexity have been proposed and  successfully applied to many problems. The disadvantage of these methods is that the prefactor can be quite large, and the memory required to store the $L$ and $U$ factors can also be large. 
In the recent years several approaches 
have been proposed for fast direct methods with HSS matrices. HSS matrix is equivalent to the $\mathcal{H}^2$-matrix corresponding to one-dimensional geometry, thus this structure is not fully suited for solving surface and volume integral equations, i.e. it can not give optimal complexity. Nevertheless, 
the actual computing times can be amazing \cite{Ho-dir_hss-2012, CorMar-dir_hss-2014,XiaSh-hss-2009}.

In this paper we use the observation that the classical three-step matrix-by-vector product procedure can be rewritten as a big linear system (which we call sparse extended form or simply SE-form of the $\mathcal{H}^2$-matrix). Very similar ideas were presented in \cite{ChanDew-hss_se-2006, Ho-dir_hss-2012, white-schur-2009, ChMing-dir_hss-2006}. In paper \cite{Amb-dir_h2-2014} SE form is used for building direct solver for system with  $\mathcal{H}^2$-matrix.  We propose a number of new methods for preconditioning systems with H2 Matrix, based on idea of SE-form. For small problem sizes (say, $N \sim 20\,000$) this gives an easy-to-implement approach for the solution of a given linear system with an $\mathcal{H}^2$-matrix, we have found that the memory requirements and the computational cost grow very fast. Therefore we propose several alternatives, which use SE-form as an auxiliary step for the creation of efficient preconditioners for the initial matrix. We numerically confirm the effectiveness of the proposed preconditioners on two surface integral equations. The code is available online as a part of the  open-source Python package h2tools \cite{h2tools}.

\section{Notations and basic facts}
In this section we recall basic definitions and notations for working with $\mathcal{H}^2$-matrices. This material is not new and can be found, for example in \cite{Borm-h2-2010}. Let us consider a matrix $A \in \mathbb{R}^{m \times n}$. First, we will need several definitions:
\begin{definition}
Block cluster tree $\mathcal{T}_r$ ($\mathcal{T}_c$) of rows (columns) of matrix $A$ is a tree where:
\begin{enumerate}
  \item Each  node $\mathcal{T}_r^t$ ($\mathcal{T}_c^s$)  is associated with some group of rows (columns)
  \item Root node  $\mathcal{T}_r^0$ ($\mathcal{T}_c^0$)  contains all rows (columns)
  \item If some group of rows (columns) is divided into subgroups, then  the corresponding node has child nodes associated with those subgroups.
\end{enumerate} \end{definition}
\begin{definition}
Let $\mathcal{T}_c$ and  $\mathcal{T}_r$ be block cluster trees of columns and rows of the matrix $A$. 
Each pair $p=( v,w)$ where $  v \in \mathcal{T}_r ,w \in \mathcal{T}_c$ represents 
some block in matrix  $A$.  We assume that there is some rule that divides the set 
of blocks into two classes, namely \emph{close blocks} and \emph{far  blocks}. The nodes of trees $i \in \mathcal{T}_c$ and $j \in \mathcal{T}_r$ are close (far) if the pair $p = (i,j)$ is close (far). 
If some node $j$ of one tree is close to all children of some node $i$ of another tree then the 
node $j$ is close to the node $i$. Denote by $\mathcal{P}_{close}$  ($\mathcal{P}_{far}$) the 
set of all close (far) pairs $p=(i,j), i \in \mathcal{T}_r ,j \in \mathcal{T}_c$
\end{definition}
\begin{definition}[Cluster basis.]{\cite[p.~54]{Borm-h2-2010}}
Let $K = (K_i)_{i \in \mathcal{T}_I}$  be a family of finite index sets. Let $R = (R_i)_{i \in \mathcal{T}_r} $ be a family of matrices satisfying $R_i \in  \mathbb{R}^{I \times K_t}$ for all $ i \in \mathcal{T}_r$. Then $R$ is called a cluster basis with rank distribution $K$ and the matrices $R_i$ are called cluster basis matrices.
\end{definition}
The $\mathcal{H}^2$-matrix structure is related to the \emph{nestedness property}.        
\begin{definition}[Nested cluster basis.]{\cite[p.~54]{Borm-h2-2010}}
Let $R$ be a cluster basis with rank distribution $K$. $R$ is called nested if there is a family 
$(E_i)_{i \in \mathcal{T}_r}$ of matrices satisfying the following conditions:
\begin{enumerate}
  \item For all $i \in \mathcal{T}_r$ and all $i' \in \sons(i)$, we have $E_{i'} \in \mathbb{R}^{K_{i'} \times K_i}$.
  \item For all $i \in \mathcal{T}_r $ with $\sons(i)$ $\neq  \varnothing  $, the following equation holds:
\begin{equation}\label{nb}  R_i = \sum_{i' \in sons(i)}R_{i'}E_{i'},\end{equation}  \end{enumerate}The matrices $E_t$ are called transfer matrices or expansion matrices.
 \end{definition}
% We can consider the matrix $A$ as a sum of two matrices. $A = C + F$ where $C$ is constructed from close blocks $p \in \mathcal{P}_{close}$ and $F$ is constructed from far blocks $p \in \mathcal{P}_{far}$. Now we can define $\mathcal{H}^2$-matrix:
%%\begin{definition}[$\mathcal{H}^2$-representation]{\cite[p.~56]{h22}}\label{hd}
%Let $A \in \mathbb{R}^{I \times J}$.  Let $\mathcal{T}_c$ and $\mathcal{T}_r$ be block cluster trees of columns and rows of matrix A. $R$ is nested cluster basis of rowe and $L$ is nested cluster basis of columns. If there is a famaly $ S = (S_p)_{p \in \mathcal{P}_{far}}$ of matrices , for all $p = (i,j) \in \mathcal{P}_{far}$ then $A$ is an $\mathcal{H}^2$-matrix and
%\begin{equation}\label{defH}
%A = F + C = \sum_{p = (i,j) \in \mathcal{P}_{far} } R_iS_pL^*_j +\sum_{p = (i,j) \in \mathcal{P}_{close} } C_p
%\end{equation}
%\end{definition}
 We can consider matrix $A$ as a sum of two matrices: $A = C + F$, where $C$ is constructed from close blocks $p \in \mathcal{P}_{close}$ and $F$ is constructed from far blocks $p \in \mathcal{P}_{far}$. Now we can define an $\mathcal{H}^2$-matrix:
\begin{definition}[$\mathcal{H}^2-representation$]{\cite[p.~56]{Borm-h2-2010}}\label{hd}
Let $A \in \mathbb{R}^{I \times J}$.  Let $\mathcal{T}_c$ and $\mathcal{T}_r$ be block cluster trees of columns and rows of matrix A. $R$ is a nested cluster basis of rows and $L$ is a nested cluster basis of columns. If there is a family $ S = (S_p)_{p \in \mathcal{P}_{far}}$ of matrices , for all $p = (i,j) \in \mathcal{P}_{far}$ then $A$ is an $\mathcal{H}^2$-matrix and
\begin{equation}\label{defH}
A = F + C = \sum_{(i,j) \in \mathcal{P}_{far} } R_iS_pL^*_j +\sum_{(i,j) \in \mathcal{P}_{close} } C_p
\end{equation}
\end{definition}
Figure~\ref{Hpic} shows close blocks (black), far blocks (white) and an example of block row (shaded). The block row consists of columns
of the blocks that are separated from the node $i$. 
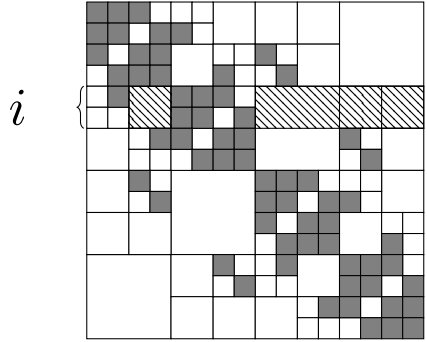
\begin{figure}[H]
\centering
\tikzset{full/.style = {fill = gray}, lowrank/.style = {}, fullrank/.style = {fill = gray}}
\begin{tikzpicture}[y = -1cm,scale = 0.28]
% rectangles of size 4
\foreach \i in {(12, 0), (8, 4), (4, 8), (0, 12)} \draw[fill = white] \i rectangle +(4, 4);
% rectangles of size 2
\foreach \i in {(6, 0), (8, 0), (10, 0), (4, 2), (10, 2), (2, 4), (12, 4), (14, 4), (0, 6), (14, 6), (0, 8), (14, 8), (0, 10), (2, 10), (12, 10), (4, 12), (10, 12), (4, 14), (6, 14), (8, 14)} \draw[fill = white] \i rectangle +(2, 2);
% rectangle of size 1
\foreach \i in {(3, 0), (4, 0), (5, 0), (2, 1), (5, 1), (1, 2), (6, 2), (7, 2), (9, 2), (0, 3), (7, 3), (8, 3), (0, 4), (7, 4), (0, 5), (1, 5), (6, 5), (2, 6), (5, 6), (13, 6), (2, 7), (3, 7), (4, 7), (12, 7), (3, 8), (11, 8), (12, 8), (13, 8), (2, 9), (10, 9), (13, 9), (9, 10), (14, 10), (15, 10), (8, 11), (15, 11), (7, 12), (8, 12), (15, 12), (6, 13), (8, 13), (9, 13), (14, 13), (10, 14), (13, 14), (10, 15), (11, 15), (12, 15)} \draw[fill = white] \i rectangle +(1, 1);
\foreach \i in {(0, 0), (1, 0), (2, 0), (0, 1), (1, 1), (3, 1), (4, 1), (0, 2), (2, 2), (3, 2), (8, 2), (1, 3), (2, 3), (3, 3), (6, 3), (9, 3), (1, 4), (4, 4), (5, 4), (6, 4), (4, 5), (5, 5), (7, 5), (3, 6), (4, 6), (6, 6), (7, 6), (12, 6), (5, 7), (6, 7), (7, 7), (13, 7), (2, 8), (8, 8), (9, 8), (10, 8), (3, 9), (8, 9), (9, 9), (11, 9), (12, 9), (8, 10), (10, 10), (11, 10), (9, 11), (10, 11), (11, 11), (14, 11), (6, 12), (9, 12), (12, 12), (13, 12), (14, 12), (7, 13), (12, 13), (13, 13), (15, 13), (11, 14), (12, 14), (14, 14), (15, 14), (13, 15), (14, 15), (15, 15)} \draw[full] \i rectangle +(1, 1);
\draw[pattern = north west lines] (2, 4) rectangle +(2, 2);
\draw[pattern = north west lines] (8, 4) rectangle +(8, 2);
\draw [thin, decorate,decoration={brace, mirror}, xshift = -4] (0, 4) -- +(0, 2) node [midway,xshift=-25, scale = 2] {$i$};
\end{tikzpicture}
\hspace{1cm}
\caption{Illustration of a matrix with $\mathcal{H}^2$-structure.}
\label{Hpic}
\end{figure}
The construction of cluster trees, block cluster nested trees, could be done using the standard procedure, see, for example, \cite{Sam-tree-1984}.  The crucial task here is to compute the cluster basis. In our numerical experiments we have used 
the method proposed in \cite{mikhos-mcbh-2013}, however other techniques maybe used as well.
Summarizing, an $\mathcal{H}^2$-matrix $A$ is defined by cluster trees $\mathcal{T}_c$ and $\mathcal{T}_r$, and the following sets of matrices $R = (R_i),~ i \in \mathcal{T}_r,$
$C = (C_p),~ p = (i,j) \in \mathbb{P}_{close},~ i \in \mathcal{T}_r,~ j \in \mathcal{T}_c,$
$S = (S_p),~ p = (i,j) \in \mathbb{P}_{far}, i \in \mathcal{T}_r,~j\in \mathcal{T}_c, $
$L = (L_j), ~j \in \mathcal{T}_c,$
\subsection{Matrix-vector multiplication}
Matrix-by-vector multiplication algorithm for the $\mathcal{H}^2$-matrix  \cite[p.~59-63]{Borm-h2-2010} plays the key role in this paper. Its formal description  is given in Algorithm \ref{mv}. 
 It is convenient to distribute the vector $x$ over the nodes of the row cluster tree:  $x = (x_i) ~i \in \mathrm{leaves}(\mathcal{T}_r).$ The resulting vector $y = Ax$ is also  computed  in the form $y = (y_j) ~j \in \mathrm{leaves}(\mathcal{T}_c).$ Note that 
 in the algorithm we use there is an additional operation that transfers $x$ to the leaves of the tree $\mathcal{T}_r$ with the 
 help of the matrices $D = (D_i), ~i \in leaves(\mathcal{T}_r),$ and $E = (E_j),~ j \in leaves(\mathcal{T}_c).$ 
\begin{algorithm}[h]
\SetKwFunction{FT}{ForwardTransformation}
\SetKwProg{myproc}{Procedure}{}{}
\SetKw{ret}{return:}
\caption{Forward transformation}
\myproc{\FT{$i, R, \hat{x}$}}{
\If{ $ \mathrm{sons}(i)$ $\neq$ 0 }{
	$\hat{x}_{father(i)} := R_i\hat{x}_i$
}
\Else{
	$\hat{x}_i := 0$\\
	\For{$j \in \mathrm{sons}(i)$}{
		$\hat{x_j}$ := \FT{$j, R, \hat{x}$}\\
		$\hat{x_i} := \hat{x_i} + R_jx_j$
	}
}
\ret{$\hat{x}$}
}
\end{algorithm} 

\begin{algorithm}[h]

\SetKwFunction{BT}{BackwardTransformation}\SetKwFunction{ret}{return:}
\SetKwProg{myproc}{Procedure}{}{}
\caption{Backward transformation}
\myproc{\BT{$i, L, \hat{y}$}}{
\If{ $ \mathrm{sons}(i)$ $\neq$ 0 }{
	$\hat{y}_{father(i)} := L_i\hat{y}_i$
}
\Else{
	\For{$j \in \mathrm{sons}(i)$}{
		$\hat{y_j} := \hat{y_j} + Ly_i$\\
		$\hat{y}$ := \BT{$j, L, \hat{y}$}
		
	}
}
\ret{$\hat{y}$}
}
\end{algorithm} 
\begin{algorithm}[H]
\label{mv}
\SetKwData{Left}{left}\SetKwData{This}{this}\SetKwData{Up}{up}
\SetKwFunction{proc}{proc}\SetKwFunction{FindCompress}{FindCompress}
\SetKwInOut{Input}{input}\SetKwInOut{Output}{output}
\caption{Matrix-vector multiplication.}

\Input {$\mathcal{H}^2$-matrix A = \{$\mathcal{T}_r$, $\mathcal{T}_c$, D, R, C, S, L, E\}, vector $x$}

\Output{vector y}

\For{$i \in\mathcal{T}_r$}{
	$\hat{x}_i := 0$
}
\For{$j \in\mathcal{T}_c$}{
	$\hat{y}_j := 0$
}

Step 1:\\
\For{$i \in\mathcal{T}_r$, $\mathrm{sons}$(i) = 0  }{
 \tcc{For all leaf nodes}
	$\hat{x}_i$  := $D_ix_i$
}
Step 2:\\
$\hat{x} := \texttt{ForwardTransformation}(root(\mathcal{T}_r), R, \hat{x})$\\
Step 3:\\
\For{$i \in\mathcal{T}_r$  }{
	\For{$j \in\mathcal{T}_c$  }{
		\If{ $(i,j) \in \mathcal{P}_{far}$ }{
			$\hat{y}_j := S_{(ij)}\hat{x}_i$
		}
		
	}
}
Step 4:\\
$\hat{y} := \texttt{BackwardTransformation}(root(\mathcal{T}_c), L, \hat{y})$\\
Step 5:\\
\For{$i \in\mathcal{T}_c$, $\mathrm{sons}(i) = 0 $ }{
	$ y_i := E_i\hat{y}_i$  
}
\SetAlgoLined
\end{algorithm}
On Figure \ref{sm} we give a graphical illustration of the matrix-vector product procedure.
\begin{figure}[h]
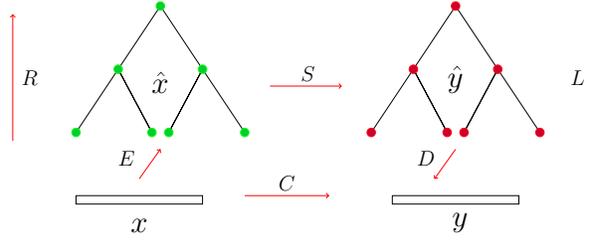

\hspace{3cm}
\resizebox{.6\textwidth}{!}{
\tikz{ \draw(0,-5.5)--(1,-4)--(1.8,-5.5)--(1,-4)--(2,-2.5)--(3,-4)--(2.2,-5.5)--(3,-4)--(4,-5.5);
    \path[fill=green!80!blue,draw=green] (0,-5.5) circle (1mm);
    \path[fill=green!80!blue,draw=green] (1,-4) circle (1mm);
    \path[fill=green!80!blue,draw=green] (1.8,-5.5) circle (1mm);
    \path[fill=green!80!blue,draw=green] (2,-2.5) circle (1mm);
    \path[fill=green!80!blue,draw=green] (3,-4) circle (1mm);
    \path[fill=green!80!blue,draw=green] (2.2,-5.5) circle (1mm);
    \path[fill=green!80!blue,draw=green] (4,-5.5) circle (1mm);
    \draw(7,-5.5)--(8,-4)--(8.8,-5.5)--(8,-4)--(9,-2.5)--(10,-4)--(9.2,-5.5)--(10,-4)--(11,-5.5);
    \path[fill=red!80!blue,draw=red] (7,-5.5) circle (1mm);
    \path[fill=red!80!blue,draw=red] (8,-4) circle (1mm);
    \path[fill=red!80!blue,draw=red] (8.8,-5.5) circle (1mm);
    \path[fill=red!80!blue,draw=red] (9,-2.5) circle (1mm);
    \path[fill=red!80!blue,draw=red] (10,-4) circle (1mm);
    \path[fill=red!80!blue,draw=red] (9.2,-5.5) circle (1mm);
    \path[fill=red!80!blue,draw=red] (11,-5.5) circle (1mm);
    \draw(0,-7)--(3,-7)--(3,-7.2)--(0,-7.2) -- (0,-7);
    \draw (1.5,-8) node[above]{\huge{$x$}} ;
    \draw [<-,  red] (2,-5.9) --(1.5,-6.6)   ;
    \draw (1.2,-6.4) node [above]  {\Large$E$} ;
    \draw (7.5,-7)--(10.5,-7)--(10.5,-7.2)--(7.5,-7.2) -- (7.5,-7);
    \draw (9,-8) node [above]  {\huge{ $y$}} ;
    \draw [<-,  red] (8.5,-6.6) -- (9,-5.9);
    \draw (8.3,-6.4) node [above]  {\Large{$D$}} ;
    \draw (11.9,-4.5) node [above]  {\Large{$L$}} ;
    \draw [<-,  red] (-1.5,-2.7) -- (-1.5,-5.7);
    \draw (5.5,-4.4) node [above]  {\Large{$S$}} ;
    \draw [<-,  red] (6.3,-4.4) -- (4.6,-4.4);
    \draw (-1.1,-4.5) node [above]  {\Large{$R$}} ;
   \draw [<-, red] (12.3,-5.7) -- (12.3,-2.7);
   \draw (5,-7) node [above]  {\Large{$C$}} ;
    \draw [<-,  red] (6,-7) -- (4,-7);
    \draw (2,-4.7) node [above]  {\huge{$\hat{x}$}} ;
    \draw (9,-4.7) node [above]  {\huge{$\hat{y}$}} ;
    }}
\caption{Illustration of the Algorithm of matrix-vector product.}
\label{sm}
\end{figure} 
The complexity is $\mathcal{O}(N)$ and storage complexity is also $\mathcal{O}(N)$.
\section{Sparse extended form: the main idea} 
Our main observation is that Algorithm \ref{mv} can be rewritten as a multiplication of
a sparse matrix by vector. The first step of Algorithm \ref{mv} can be rewritten as a matrix-vector product $Dx = \hat{x}_l$, where 
\begin{equation} \label{Dmat}
D = \begin{bmatrix}
D_1 & 0 & 0 &0 \\ 
0 & D_2 & 0 & 0\\ 
 0&  0&  \ddots &0 \\ 
 0&0  &  0& D_{N_{lr}}
\end{bmatrix}, 
\end{equation}
where $N_{lr}$ is number of leaf nodes in tree $\mathcal{T}_r$, zeros represent zero matrices of appropriate sizes.
We can rewrite  the Step 2 as $R\hat{x} = \hat{x}_{n}$, where
\begin{equation} \label{Rmat}
R = \begin{bmatrix}
R_1 &  0&0  &0 \\ 
0 & R_2 & 0 & 0\\ 
 0&0  &  \ddots & 0\\ 
 0& 0 &0  & R_{N_{nr}}
\end{bmatrix}, 
\end{equation}
where $N_{nr}$ is the number of non-leaf nodes in the tree $\mathcal{T}_r$ zeros are zero matrices of appropriate sizes. 
The third step of Algorithm \ref{mv} corresponds to equation $ L\hat{y}_{n}+S\hat{x} =  \hat{y} $, where 
\begin{equation} \label{LCmat}
L = \begin{bmatrix}
L_1 &0  & 0 &0 \\ 
 0& L_2 & 0 & 0\\ 
 0&0  &  \ddots &0 \\ 
 0&0  &0  & L_{N_{lc}}
\end{bmatrix} , S = \begin{bmatrix}
S_{11} &S_{12}   & \cdots &S_{1N_c}  \\ 
 S_{21} & S_{22} & \cdots & S_{2N_c}\\ 
 \vdots& \vdots &  \ddots & \\ 
 S_{N_c1} &  S_{N_c2}&  & S_{N_cN_c}
\end{bmatrix} ,
\end{equation}
where $N_c$ is the number of nodes in tree $\mathcal{T}_c$,  $N_{lc}$ is the number of leaf nodes in tree $\mathcal{T}_c$,  $S_{ij} = 0$ if $i\in\mathcal{T}_c$  $j \in \mathcal{T}_r$ $(i,j) \in  \mathcal{P}_{far}$. And the final step we rewrite as  $y = Ey_l + Cx$, where 
\begin{equation} \label{Emat}
E = \begin{bmatrix}
E_1 & 0 & 0 & 0\\ 
 0& E_2 & 0 & 0\\ 
 0& 0 &  \ddots &0 \\ 
 0&0  &0  & E_{N_{lc}}
\end{bmatrix},
\end{equation}
Putting this all together we get
\begin{equation} \label{system}
\left\{\begin{matrix}
 Dx = \hat{x}_l \\
 R\hat{x} = \hat{x}_{n} \\
 L\hat{y}_{n}+S\hat{x} =  \hat{y}  \\
 E\hat{y}_{l}+Cx  = y
 \end{matrix}
\right.,
\end{equation}
or in the block form:
\begin{equation} \label{l1}
\begin{bmatrix}
C & 0 & 0 &E \\ 
 0& S &L  &0 \\ 
 0&  R& 0 & 0\\ 
 D& 0 & 0 & 0
\end{bmatrix} 
\begin{bmatrix}
     x\\ \hat{x}\\ \hat{y}_{n}\\  \hat{y}_{l} 
\end{bmatrix} = 
\begin{bmatrix}
     y\\ \hat{y}\\ \hat{x}_{n}\\  \hat{x}_{l} 
\end{bmatrix},\end{equation}
Denote the obtained matrix by $H_0$, and also introduce a block vector $\hat{x} = \begin{bmatrix}
      \hat{x}_n \\ 
      \hat{x}_l 
  \end{bmatrix}$ and $\hat{y} = \begin{bmatrix}
      \hat{y}_n \\ 
      \hat{y}_l 
  \end{bmatrix},$.
  Finally, we get
\begin{equation}\label{l2}
  H_0 \begin{bmatrix}
      x  \\ 
      \hat{x}\\
      \hat{y}  
  \end{bmatrix} = \begin{bmatrix}
      y  \\ 
      \hat{y}\\
      \hat{x}  
  \end{bmatrix}.
\end{equation}
Recall that our goal is given $y$ compute $x$, therefore
$$\left ( H_0  +
\begin{bmatrix}
      0^{N\times N} &0&0 \\ 
        0 &0&-I^{N_{\hat{y}}\times N_{\hat{y}}} \\ 
        0 &-I^{N_{\hat{x}}\times N_{\hat{x}}}&0 \\  
  \end{bmatrix}
  \right )  \begin{bmatrix}
      x  \\ 
      \hat{x}\\
      \hat{y}  
  \end{bmatrix} = \begin{bmatrix}
      y  \\ 
      0\\
      0  
  \end{bmatrix} ,
$$
where $N_{\hat{x}} = len(\hat{x})$, $N_{\hat{y}} = len(\hat{y})$ and 
 $$H = H_0  +\begin{bmatrix}
      0 &0&0 \\ 
        0 &0&-I \\ 
        0 &-I&0 \\  
  \end{bmatrix}
.$$
The final system of equations has the form
\begin{equation}\label{l3}
  H \begin{bmatrix}
      x  \\ 
      \hat{x}\\
      \hat{y}  
  \end{bmatrix} = \begin{bmatrix}
      y  \\ 
      0\\
       0 
  \end{bmatrix}.
\end{equation}
Now the right hand side of system \eqref{l3} contains only known values and 
we can solve it and find $x$, where $x$ is the solution of $Ax = y$.
We will call the matrix $H$ ``sparse extended form'', or SE-form, of the $\mathcal{H}^2$-matrix $A$. 
\section{Properties of the SE-form}
An important property of the SE-form is that if $A$ non-singular, 
the SE-form of $A$ is non-singular as well.
\begin{theorem}
If a matrix $A \in \mathbb{R}^{N \times N}$  is a nonsingular $\mathcal{H}^2$-matrix, then  $H = \SE(A) \in \mathbb{R}^{N_H \times N_H}$ is nonsingular, and   $N_H < (2k+1)N$, where $k$ is the maximum of numbers of levels of the cluster trees $\mathcal{T}_c$ 
and  $\mathcal{T}_r$.
\end{theorem}
\begin{proof}
First let us prove that the matrix $H$ is square. The matrix $A$ is square therefore  $\len(x)=\len(y)= N$. 
Let $N_r$ be the number of rows in $H$ and $N_c$ be the number of columns in it. 
Then, 
$$N_r = \len(x)+\len(\hat{x})+\len(\hat{y}) = \len(y) + \len(\hat{y}) + \len(\hat{x})= N_c  $$
Thus $H$ is a square  matrix. Note $\len(\hat{x}) \leq k_1N$, $\len(\hat{y}) \leq k_2N$, where $k_1$ and $k_2$ are the 
numbers of levels  of the trees  $\mathcal{T}_c$ and  $\mathcal{T}_r$. Therefore
$$N_H =  \len(x)+\len(\hat{x})+\len(\hat{y}) = N + k_1N + k_2N < (2k+1)N$$
Now we prove that $H$ is nonsingular. 
Suppose that $Hz = 0$ and let us prove that it implies $z = 0$.
Due to the construction of the SE-form, the first block component of the vector $z$ satisfies 
$Ax = 0$, thus $x = 0$. 
According to Algorithm \ref{mv} and  \eqref{system} $x = 0$ leads to  $\hat{x} = 0$ and  $\hat{y} = 0$
therefore $z = 0$ and the kernel of $H$ is trivial.
\end{proof}
Note, that the condition number of the SE-form is typically much larger than of the original matrix, so 
special preconditioning is need
\section{Solvers based on SE-form}
\label{TP}
How we can use the SE-form of the matrix for the solution of linear systems. We propose several methods, which are listed
below.
\begin{enumerate}[label=\bfseries Method \arabic*:,leftmargin=*]
    \item (Direct solver)
Apply sparse direct solvers to $\SE(A)$ and given $y$ compute $x$.
 
  \item  (Preconditioning \eqref{l3} with matrix $\SE(A)$)
Construct preconditioner to $\SE(A)$ based on the block structure.
  \item  (Iterative solvers for systems with $A$ using $\SE(A)$ as preconditioner)
SE-form can be used as preconditioner for the original system. To solve the correction equation we apply several steps of some solver for $\SE(A)$.
\end{enumerate}
Now we describe them in more details. 
\subsection{Method 1}
Applying any fast sparse direct solver to $\SE(A)$ is natural idea. However, it is ok for small $N$, but the amount of memory required for the such solver grows very quickly. The advantage of this method is that it is very simple to implement. Once a sparse 
representation of the $\SE(A)$ is given, we only need to call the procedure. Another approach is to compute some preconditioner
for $\SE(A)$. In our numerical experiments we have tested ILUT preconditioner. We will call this approach SE-ILUT preconditioner.
\subsection{Method 2}
\label{M2}
Now we consider system \eqref{l3} with sparse extended matrix $\SE(A)$ and find solution of this system iteratively.  In numerical experiments we found that $\SE(A)$ always have very large condition numbers. Thus, a preconditioner is needed.
A natural way is to use the block structure of $\SE(A)$ to construct the preconditioner. 
We propose a \textbf{SE-Block preconditioner}. We compute an approximate inverse of the ``far block'' of $\SE(A)$ and all others
are replaced by identity matrices:
\begin{equation} \label{Bp}
B = \begin{bmatrix}
I & 0 & 0  \\ 
0 & P(S) & 0  \\ 
 0 & 0 &  I 
\end{bmatrix}, 
\end{equation}
 where $P(S)$ is some preconditioner for the block $S$. Note, that block $S$ can be rectangular, in 
this case  we construct preconditioner for the smallest  square block that contains $S$.
In experiments we have seen that the ill-conditioning of $S$ is the reason why $H$ is ill-conditioned.
\subsection{Method 3}
\label{subsec:hp}
The matrix $A$ can be considered as a Schur complement of $H = \SE(A)$ with components related to $\widehat{x}$ and
$\widehat{y}$ eliminated. Typically, Schur complement is used as a preconditioner; 
here we use \emph{reverse Schur component preconditioning} (similar ideas were used in \cite{white-schur-2009}).
In this method, Schur complement is used in the opposite fashion: we solve the system with a matrix $A$ 
and to solve the correction equation we go to the large system (extend the residual), apply several 
steps of some preconditioned iterative solver for $\SE(A)$ and extract the required vector as a corresponding component
of the result. This approach appears to be the most effective one for the problems we have considered. 
Moreover, additional speedup can be obtained by using \emph{recompression} of the $\mathcal{H}^2$-matrix, i.e.
$$A \approx B$$
where $B$ has smaller ranks in the rank distribution. The description of the robust algorithm based
on SVD can be found in \cite{Borm-h2-2010}, and it is implemented in the h2tools package as well. 
Then we use $\SE(B)$ instead of $\SE(A)$ in the algorithm, so this method is not a ``reverse Schur complement'', 
but ``approximate reverse Schur complement'' method. We will denote this approach \emph{SVD-SE}.
The final algorithm is summarized in Algorithm~\ref{svd-se-schur}.
%Basic idea here is  that SE-form can be used as preconditioner for the original system. Consider $\mathcal{H}^2$-matrix  $A \in \mathbb{R}^{N\times N}$, $H = \SE(A) \in \mathbb{R}^{N_H\times N_H}$, $\mathcal{H}^2$-MatVec(x) is function that perform $\mathcal{H}^2$-product of matrix $A$ by vector $x$.  We propose the following Algorithm.
%\begin{algorithm}[H]
%\label{PH2}
%\SetKwData{Left}{left}\SetKwData{This}{this}\SetKwData{Up}{up}
%\SetKwFunction{pr}{Preconditioner}\SetKwFunction{is}{ItSol}
%\SetKwInOut{Input}{input}\SetKwInOut{Output}{output}
%\caption{SVD-SE}
%\SetKwFunction{}{}
%\SetKwProg{myproc}{Procedure}{}{}
%\SetKw{ret}{return:}
%\Input{$A$, $y$, $\mathcal{H}^2$-MatVec(x) }
%\Output{$x$}
%Run GMRES method for the matrix $A$ with the fast $\mathcal{H}^2$ matvec  
%x := ItSol($\mathcal{H}^2$-MatVec(x), $y$, \pr{$y$}) \\
%\myproc{\pr{$y$}}{
% $Y := Extend(y)$\\
% $X := ApprSol(H,Y)$\\
% $x := Compress(X)$\\
% \ret{$x$}}

%\SetAlgoLined
%\end{algorithm}
\begin{algorithm}[H]
\DontPrintSemicolon
\SetKwComment{tcp}{$\triangleright$~}{}
\KwData{Matrix $A$ in the $\mathcal{H}^2$-format, right-hand side $y$, required tolerance 
$\varepsilon$, inner parameters: $\delta_{\mathrm{ILUT}}$ and $\delta_{\mathrm{SVD}}$, number of reverse Schur iterations $k_{\mathrm{schur}}$ }
\KwResult{Approximate solution $x$ such that $\Vert A x - y \Vert \leq \varepsilon$}
\Begin{
    \nl Compute $B$ as a recompressed $\mathcal{H}^2$ representation of $A$ with accuracy $\delta_{\mathrm{SVD}}$ %\tcp{Initialization}

    \nl Compute $\SE(B)$ and ILUT preconditioner $P$ with drop tolerance $\delta_{\mathrm{ILUT}}$

    \nl $x = \mathrm{GMRES}(A, y, \mathrm{tol}=\varepsilon, \mathrm{prec}=\mathrm{RevSchur})$
    %
%    \nl Run GMRES method for the the matrix $A$ and required accuracy $\varepsilon$ and SErevSchur
    
    \setcounter{AlgoLine}{0}
    \SetKwFunction{RevSchur}{RevSchur}
    \SetKwProg{myproc}{Procedure}{}{}
    \myproc{\RevSchur{$r$}}{
        \nl Extend right-hand side: 
    $\widehat{r} = \begin{pmatrix} r^{\top} & 0 & 0 \end{pmatrix}^{\top}$

              \nl Do $k_{\mathrm{schur}}$ steps of ILUT-preconditioned GMRES for $\SE(B)$:  
        $
        \widehat{z} = 
         \mathrm{GMRES}(\SE(B), \widehat{r}, \mathrm{maxit}=k_{\mathrm{schur}}, \mathrm{prec}=P) 
        $

        \nl Extract $z$ as the first $N$ components of $z$

        %    \nl Do $k_{\mathrm{schur}}$ iterations of GMRES for $\SE(B)$ using $P$ as preconditioner and 
    %$z = \begin{bmatrix} x \\ 0 \\ 0 \end{bmatrix}$ as right-hand side.
    
    \nl \KwRet\ $z$}
    }
\caption{SVD-SE method}\label{svd-se-schur}
\end{algorithm}

%The function $ItSol()$ here is some iterative solver that usually require matrix-vector product procedure, right hand side vector and optionally preconditioner. For example it can be GMRES \cite{gmres} or BiCGStab \cite{bi} solver. The function Y := Extend(y) extend vector y by zeros, where $Y\in\mathbb{R}^{N_H}$, $y\in\mathbb{R}^{N}$. The function ApprSol() finding approximate solution of equation of equetion $HX = Y$, it can be for example a few number of iterations of ILUt algorithm \cite{ilut}. The function x = Compress(X) extracts vector $x\in\mathbb{R}^{N}$ from vector $X = \begin{bmatrix} x \\ \hat{x} \\ \hat{y} \end{bmatrix}$, where $X\in\mathbb{R}^{N_H}$. 

%It is possible to build approximation of matrix $\SE(A)$. We compute low-rank approximations for far matrices $S_i$, then we recompute transition matrices $R_i$, $L_i$, $E_i$, $D_i$. The next step is to build SE-matrix from new approximated blocks, we obtain some new compressed SE-form. Since we compute low-rank approximation using SVD decomposition, let us denote new matrix SVD-SE(A). In the Algorithm \ref{PH2} we can use matrix  SVD-SE(A) instead of matrix $\SE(A)$.
\section{Numerical experiments}
\subsection{Electrostatic problem}
As a model problem we consider boundary integral equation of the form
 \begin{equation} \label{ie1}
     \int\limits_{\Omega}\frac{q(y)}{\left 
\| x-y \right \|}dy=f(x), \quad x \in \Omega,
\end{equation} where $\Omega$ is  $[0, 1]^2$. The function $f(x)$ is given and $q(y)$ is to be found. 
Equation \eqref{ie1} is discretized using collocation method  with piece-wise constant basis functions 
on the triangular mesh $\Omega_N$ with $N$ triangles (see Figure~\ref{sutf1}) 
The matrix elements can be evaluated analytically \cite{Hess-quad-1962}.
The matrix $A$ is then approximated in the $\mathcal{H}^2$-format using the h2tools package \cite{h2tools}.  
The $\SE$-form of $A$ is presented on Figure~\ref{SEA} for $N = 1196$. 
\begin{figure}[H]
\begin{minipage}[H]{0.55\linewidth}
\center{\includegraphics[scale = 0.45]{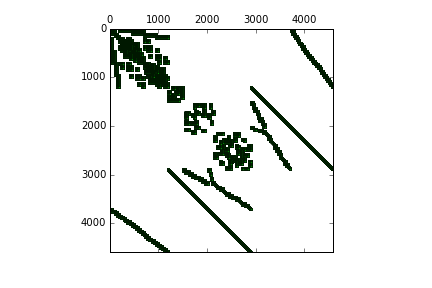}
%\scalebox{0.39}{\input{SE_sq.pgf}}
\caption{SE-form of matrix $A$}
\label{SEA}}
\end{minipage}
\hspace{0.5cm}
\begin{minipage}[H]{0.34\linewidth}
%\vspace{-0.4cm}
\center{\includegraphics[scale = 0.15]{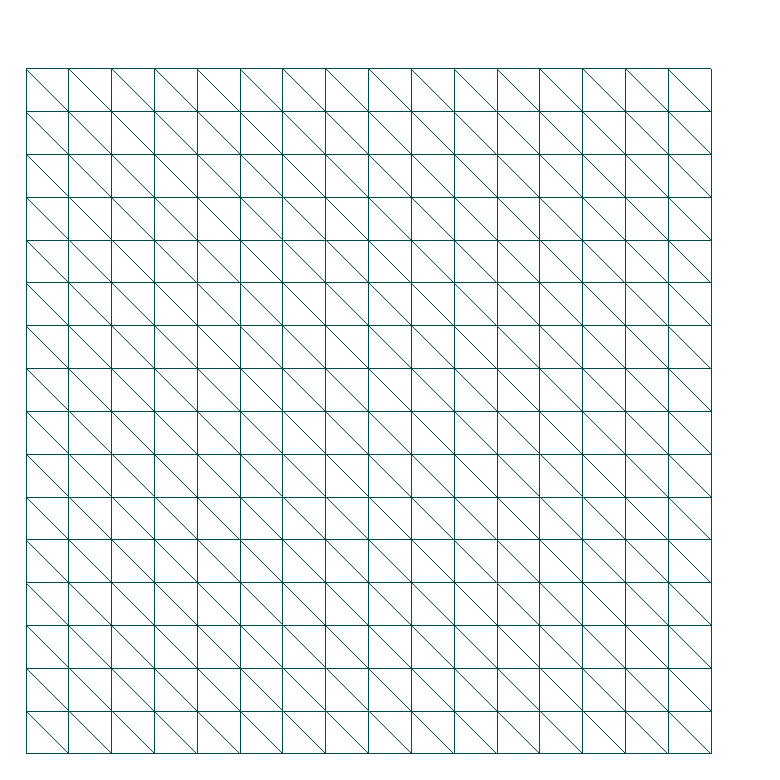}
\vspace{0.4cm}
\caption{Triangular mesh $\Omega_N$}
\label{sutf1}}
\end{minipage}
\end{figure}
\subsubsection{Method 1}
In Table \ref{tab:dir} the results of Method 1 (sparse direct solver applied to $\SE(A)$) are presented. 
As it is readily seen, the memory quickly becomes a bottleneck.
\begin{figure}[H]
\begin{center}
  \begin{tabular}{ | l | c  c | }
    \hline
N & time, (s) & Mem, (MB)\\ \hline
    3928	& 2.585 & 376.85	  \\ %\hline 
	13640	& 37.76 & 2527.7	 \\ %\hline 
    28080 	& 234  	& 5012.1	\\ %\hline 
    59428 	& ---    	& ---	\\ %\hline 
    98936 	& ---     &  ---	 \\ %\hline 
    \hline
  \end{tabular}
\end{center}
\caption{Timings and memory for Method 1}
\label{tab:dir}
\end{figure}

\subsubsection{Method 2}
The convergence of GMRES method with SE-ILUt and block preconditioners is presented on Figure~\ref{SE}. 
%Note that in the Method 2 we perform GMRES iterations with matrix $\SE(A)$. Despite  $\SE(A)$-matrix is usually ill conditioned, GMRES iterations with SE-ILUt or  Block preconditioner are convergent. Comparison of residual decline speed for GMRES iterations with  SE-ILUt preconditioner,  Block preconditioner and without  preconditioner is presented in the Figure \ref{SE}. 
\begin{figure}[H]
\begin{center}
\resizebox{0.9\linewidth}{!}{
\scalebox{0.3}{\input{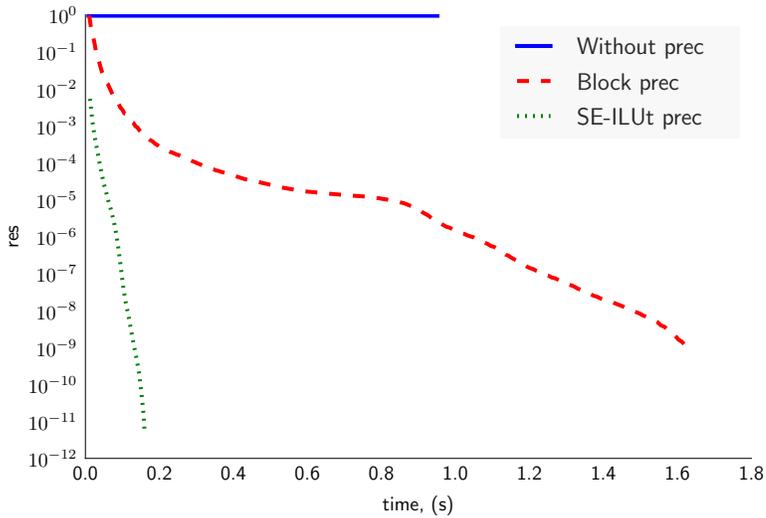}}}
\caption{Method 2: GMRES convergence with different preconditioners}
\label{SE}
\end{center}
\end{figure}
The number of iterations with the SE-ILUT preconditioner is significantly smaller than the 
number of iterations with the block preconditioner, however, the computation of the block preconditioner is much less 
expensive. This is illustrated in Table~\ref{tab:SE}.   
%In the Table \ref{tab:SE} were present timing for building preconditioners plus time for GMRES solution, droptol = $10^{-8}$.
\begin{figure}[H]
\begin{center}
  \begin{tabular}{ | l |  c  c |   }
    \hline
N	& Block prec + GMRES, (s) 	& SE-ILUt prec + GMRES, (s)	\\ \hline
    3928		& 0.085 + 1.28 	& 1.75 + 0.17	\\  
	13640	  	& 0.23 + 5.6 	& 9.17 + 0.52	\\ %\hline 
    28080		& 0.53 + 11.8 	& 27.17 + 0.91	\\ %\hline 
    59428  		& 1.34 + 34.8 	& 75.02+ 3.13	\\ %\hline 
    98936		& 3.28 + 59.13 & 134.11+ 10.2	\\ %\hline 
    \hline
  \end{tabular}
\end{center}
\caption{Timing for building the preconditioner and solving the system using the GMRES method}
\label{tab:SE}
\end{figure}
\subsubsection{Method 3} 
The convergence of the GMRES method with reverse-Schur preconditioning is presented on Figure~\ref{H2} for $N = 28080$.
\begin{figure}[H]
\begin{center}
\resizebox{0.9\linewidth}{!}{
\scalebox{0.3}{\input{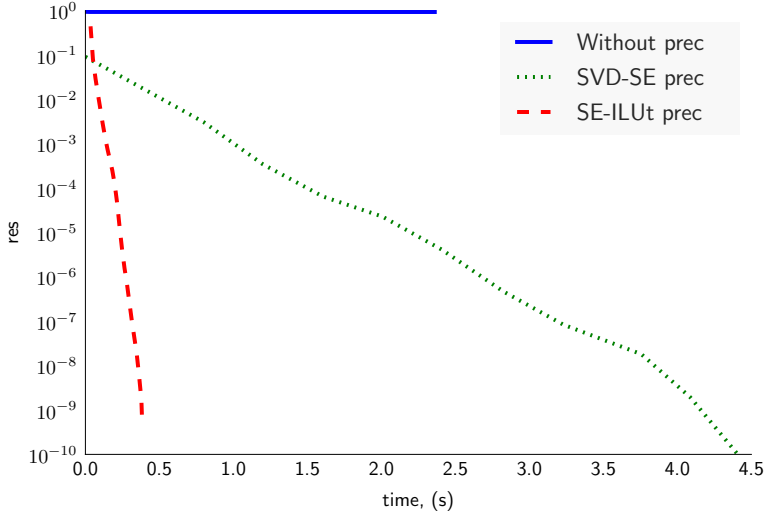}}}
\caption{Method 3: Convergence of GMRES for $N = 28080$}
\label{H2}
\end{center}
\end{figure}
The total computational cost for different reverse-Schur preconditioner is given in Table~\ref{tab:H2}. Note that
the SE-ILUt preconditioner in the second column is the same as in the previous section, however it is more effective to use
it as a reverse-Schur preconditioner for the original system, than for the full $\SE$-form.
The compression of the $\mathcal{H}^2$-form of the matrix yields the best preconditioner by a factor of $4$. 
%As mentioned above to know real solution time we have to take into account  preconditioner building time. In the Table \ref{tab:H2} were present timing for building preconditioners plus time for GMRES solution, droptol = $10^{-8}$.
\begin{table}[H]
\begin{center}
  \begin{tabular}{ | l | c   c |  }
    \hline
N 	& SE-ILUt prec + GMRES&  SVD-SE prec + GMRES, (s)	\\ \hline%\hline 
    3928		& 2.13 + 0.03	&  0.21 + 0.11 	\\ 
	13640		& 8.84 + 0.11	&  1.34 + 1.32		\\ %\hline 
    28080	 	& 24.8 + 0.35	&  8.35 + 2.94		\\ %\hline 
    59428 		& 69.1 + 1.33	&  19.7 + 6.13		\\% \hline 
    98936  		& 150.2 + 4.38  &  40.01 + 15.03			\\ %\hline 
    \hline
  \end{tabular}
\end{center}
\caption{Timings for the Method 3 with different reverse-Schur preconditioners}
\label{tab:H2}
\end{table}
\subsubsection{Final comparison}
In Table~\ref{tab:ft} we present final comparison for all methods for different $N$. 
For Method 2 and Method 3 the solution time is the total computational time (building the preconditioner and solving the system 
using GMRES) 
%we consider solution time as sum of preconditioner building time and GMRES iteration time, droptol = $10^{-8}$.
\begin{table}[H]
\begin{center}
  \begin{tabular}{ | l | c  c  c  c  c |  }
    \hline
& Method 1 & \multicolumn{2}{ c }{Method 2}& \multicolumn{2}{ c| }{Method 3}  \\ \cline{2-6}
	N & Dir. sol, (s)&  Block, (s) & SE-ILUt, (s) 	& SE-ILUt, (s) & SVD-SE, (s)  \\ \hline
    3928	& 2.585 	& 3.215 	& 2.11 	& 2.16		& 0.87 	\\ 
    13640	& 37.76  	& 10.83 	& 9.69 	& 8.95		& 4.65	\\  
    28080 	& 234  		& 22.33 	& 28.08	& 25.47 	& 16.92	\\  
    59428 	& ---   		& 53.14 	& 78.15	& 70.37 	& 42.01	\\  
    98936 	& ---      	& 122,41 	& 144.31	& 127.95 	& 89.94	\\  
    \hline
  \end{tabular}
\end{center}
\caption{Timings for all methods for different $N$.}
\label{tab:ft}
\end{table}
The memory is an important constraint, and in Table~\ref{tab:fm} we present the memory required for each of the methods for different 
$N$. 
\begin{figure}[H]
\begin{center}
\resizebox{0.95\linewidth}{!}{
  \begin{tabular}{ | l | c  r  c  c  c |  }
    \hline
& Method 1 & \multicolumn{2}{ c }{Method 2}& \multicolumn{2}{ c| }{Method 3}  \\ \cline{2-6}
	N & Dir. sol,\,(MB)&Block,\,(MB)  	& SE-ILUt,\,(MB) 	& SE-ILUt,\,(MB) & SVD-SE,\,(MB)  \\ \hline
    3928	& 424		& 24.2 		& 31.4 		& 15.7		& $<$ 10	\\  
    13640	& 2606  	& 251.2 	& 279.8 	& 31.4		& $<$	10	\\  
    28080 	& 14702		& 675.3 	& 2464.9	& 219.8		& 47.1	\\  
    59428 	& ---   		& 1301		& 9720.6	& 810 		& 128.3	\\  
    98936 	& ---      	& 6340.1 	& 	---	  	& 2028.8 	& 221.5	\\  
    \hline
  \end{tabular}}
\end{center}
\caption{Memory requirements for all the methods for different $N$, missing entries mean ``out of memory''}
\label{tab:fm}
\end{figure}
Method 3 with SVD recomporession is the fastest method and requires much less memory.
\subsection{Hypersingular integral equation}
The second problem is the hypersingular integral equation 
\begin{equation} \label{ie2}
\int\limits_{\Omega}\frac{q(y)}{\left 
\| x-y \right \|^3}\, dy=f(x),
\end{equation} 
where $\Omega$ is shown on Figure~\ref{sutf2}.
The equation is discretized using the collocation method with piecewise-constant basis functions (also known as 
discrete vortices method), and the approximation in the $\mathcal{H}^2$-format is done using the h2tools package as well.
The $\SE$-form of $A$ is given on Figure~\ref{SE_a}. The mesh has $N = 28340$ triangles.
\begin{figure}[H]
\begin{minipage}[h]{0.46\linewidth}
\vspace{2cm}
\center{\includegraphics[scale = 0.39]{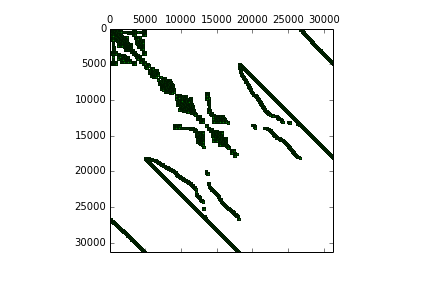}
\vspace{0.4cm}
\caption{SE-form of matrix $A$}
\label{SE_a}}
\end{minipage}
\begin{minipage}[H]{0.49\linewidth}
\center{\includegraphics[scale = 0.23]{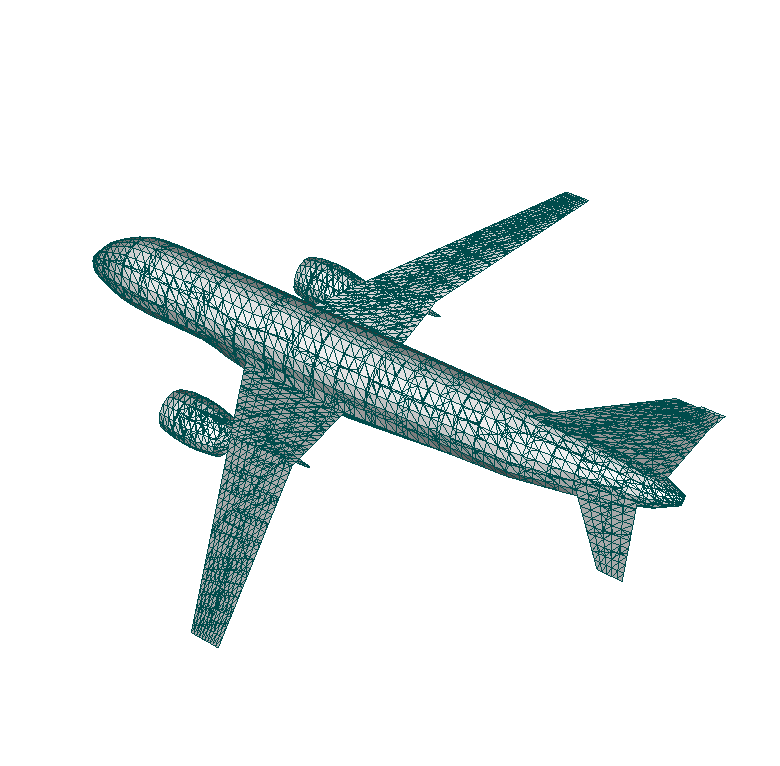}
\caption{Surface $\Omega_N$}
\label{sutf2}}
\end{minipage}
\hspace{0.5cm}
\end{figure}
The comparison of different methods is shown in Table~\ref{tab:air}.
\begin{figure}[H]
\begin{center}
  \begin{tabular}{ | l | c  c  c  c  c |  }
    \hline
& Method 1 & \multicolumn{2}{ c }{Method 2}& \multicolumn{2}{ c| }{Method 3}  \\ \cline{2-6}
	 & Dir. sol &  Block  	& SE-ILUt 	& SE-ILUt & SVD-SE  \\ \hline
    memory, (M)	& -	& 28.3	& 31.7 	& 31.7	& 17.3	\\  
    time, (s)	& -  & -	& - 	& 8.28	& 13.4	\\     
    \hline
  \end{tabular}
\end{center}
\caption{Timing and memory for all methods for problem \eqref{ie2}}
\label{tab:air}
\end{figure}

\section{Conclusions}
The new SE-form of the $\mathcal{H}^2$-matrix allows for different ways for the constuction of effective linear 
systems solvers with such matrices. Numerical experiments show that the most effective way is to use
the reverse-Schur preconditioner with SVD-recompression of the $\mathcal{H}^2$-form of $A$.
The implementation of the methods is available as a part of the open-source package h2tools \cite{h2tools}, and the numerical experiments 
are available as IPython notebooks.
%In this paper we proposed family of methods for fast solution of boundary integral equations. These methods are based on the sparse-extended representation of an $\mathcal{H}^2$-matrix. 
%In numerical experiments it was shown that these methods reduce time and memory costs of iterative solution for ill-conditioned problems. Thus the new methods can be helpful for solving  some  challenging boundary integral equations.  All presented experiments are available in package h2tools \cite{h2tools}.
\bibliography{lib}

\begin{thebibliography}{10}

\bibitem{Amb-dir_h2-2014}
{\sc S.~Ambikasaran and E.~Darve}, {\em The inverse fast multipole method},
  {arXiv} preprint, 2014.

\bibitem{white-schur-2009}
{\sc J.~Bardhan, M.~Altman, B.~Tidor, and J.~White}, {\em {``Reverse-Schur''}
  approach to optimization with linear {PDE} constraints: Application to
  biomolecule analysis and design}, J. Chem. Theory Comput., 5 (2009),
  pp.~3260--3278.

\bibitem{Beb-hlu-2005}
{\sc M.~Bebendorf}, {\em Hierarchical {LU} decomposition-based preconditioners
  for {BEM}}, Computing, 74 (2005), pp.~225--247.

\bibitem{Borm-h2-2010}
{\sc S.~B{\"o}rm}, {\em Efficient numerical methods for non-local operators:
  H2-matrix compression, algorithms and analysis}, vol.~14, European
  Mathematical Society, 2010.

\bibitem{Borm-h-2003}
{\sc S.~Borm, L.~Grasedyck, and W.~Hackbusch}, {\em {Introduction to
  hierarchical matrices with applications}}, Eng. Anal. Bound Elem., 27 (2003),
  pp.~405--422.

\bibitem{ChanDew-hss_se-2006}
{\sc S.~Chandrasekaran, P.~Dewilde, M.~Gu, W.~Lyons, and T.~Pals}, {\em A fast
  solver for {HSS} representations via sparse matrices}, SIAM J. Matrix Anal.
  Appl., 29 (2006), pp.~67--81.

\bibitem{ChLi-hss-2007}
{\sc S.~Chandrasekaran, M.~Gu, X.~Li, and J.~Xia}, {\em Some fast algorithms
  for hierarchically semiseparable matrices}, Private Communication,  (2007).

\bibitem{ChLyons-hss-2005}
{\sc S.~Chandrasekaran, M.~Gu, and W.~Lyons}, {\em A fast adaptive solver for
  hierarchically semiseparable representations}, Calcolo, 42 (2005),
  pp.~171--185.

\bibitem{ChMing-dir_hss-2006}
{\sc S.~Chandrasekaran, M.~Gu, and T.~Pals}, {\em A fast ulv decomposition
  solver for hierarchically semiseparable representations}, SIAM J. Matrix
  Anal. A., 28 (2006), pp.~603--622.

\bibitem{CorMar-dir_hss-2014}
{\sc E.~Corona, P.-G. Martinsson, and D.~Zorin}, {\em An {$O(N)$} direct solver
  for integral equations on the plane}, Appl. Comput. Harmon. A.,  (2014).

\bibitem{GrRo-fmm-1987}
{\sc L.~Greengard and V.~Rokhlin}, {\em A fast algorithm for particle
  simulations}, J. Comput. Phys., 73 (1987), pp.~325--348.

\bibitem{GrRo-fmm-1988}
{\sc L.~Grengard and V.~Rokhlin}, {\em The rapid evaluation of potential fields
  in three dimensions}, Springer, 1988.

\bibitem{HackBorm-h2-2002}
{\sc W.~Hackbusch and S.~Borm}, {\em $\mathcal{H}^2$-matrix approximation of
  integral operators by interpolation}, Appl. Numer. Math., 43 (2002),
  pp.~129--143.

\bibitem{Hess-quad-1962}
{\sc J.~L. Hess and A.~Smith}, {\em Calculation of non-lifting potential flow
  about arbitrary three-dimensional bodies}, tech. rep., DTIC Document, 1962.

\bibitem{Ho-dir_hss-2012}
{\sc K.~Ho and L.~Greengard}, {\em A fast direct solver for structured linear
  systems by recursive skeletonization}, SIAM J. Sci. Comput., 34 (2012),
  pp.~A2507--A2532.

\bibitem{Mar-hss-2011}
{\sc P.~G. Martinsson}, {\em A fast randomized algorithm for computing a
  hierarchically semiseparable representation of a matrix}, SIAM J. Matrix
  Anal. A., 32 (2011), pp.~1251--1274.

\bibitem{h2tools}
{\sc A.~Mikhalev}.
\newblock \url{https://bitbucket.org/muxas/h2tools}, 2013.

\bibitem{mikhos-mcbh-2013}
{\sc A.~Mikhalev and I.~V. Oseledets}, {\em Adaptive nested cross approximation
  of non-local operators}, {arXiv} preprint 1407.1572, 2013.

\bibitem{Ro-fmm-1985}
{\sc V.~Rokhlin}, {\em Rapid solution of integral equations of classical
  potential theory}, J. Comput. Phys., 60 (1985), pp.~187--207.

\bibitem{Sam-tree-1984}
{\sc H.~Samet}, {\em The quadtree and related hierarchical data structures},
  ACM Comput. Surv., 16 (1984), pp.~187--260.

\bibitem{Sent-electr_ex-1992}
{\sc S.~D. Senturia, R.~M. Harris, B.~P. Johnson, S.~Kim, K.~Nabors, M.~A.
  Shulman, and J.~K. White}, {\em A computer-aided design system for
  microelectromechanical systems (memcad)}, Mater. Res. Soc. Symp. P., 1
  (1992), pp.~3--13.

\bibitem{ShDewilde-hss-2007}
{\sc Z.~Sheng, P.~Dewilde, and S.~Chandrasekaran}, {\em Algorithms to solve
  hierarchically semi-separable systems}, in System theory, the Schur algorithm
  and multidimensional analysis, Springer, 2007, pp.~255--294.

\bibitem{tee-mosaic-1996}
{\sc E.~E. Tyrtyshnikov}, {\em Mosaic-skeleton approximations}, Calcolo, 33
  (1996), pp.~47--57.

\bibitem{XiaSh-hss-2009}
{\sc J.~Xia, S.~Chandrasekaran, M.~Gu, and X.~S. Li}, {\em Superfast
  multifrontal method for large structured linear systems of equations}, SIAM
  J. Matrix Anal. A., 31 (2009), pp.~1382--1411.

\end{thebibliography}
\end{document}